\documentclass[11pt]{amsart}
\usepackage{amssymb}
\usepackage{graphics}
\usepackage{latexsym}
\usepackage{amsmath}
\usepackage{amssymb,amsthm,amsfonts}
\usepackage{mathtools}
\usepackage{amscd}
\usepackage[arrow, matrix, curve]{xy}
\usepackage{syntonly}
\title[Inverse Cartier transform via exponential]{An inverse Cartier transform via exponential in positive characteristic}
\author[Guitang Lan]{Guitang Lan}
\email{lan@uni-mainz.de}
\address{Institut f\"{u}r  Mathematik, Universit\"{a}t
Mainz, Mainz, 55099, Germany}

\author[Mao Sheng]{Mao Sheng}
\email{msheng@ustc.edu.cn}
\address{School of Mathematical Sciences,
University of Science and Technology of China, Hefei, 230026, China}
\author[Kang Zuo]{Kang Zuo}
\email{zuok@uni-mainz.de}
\address{Institut f\"{u}r  Mathematik, Universit\"{a}t
Mainz, Mainz, 55099, Germany}

\begin{document}

\theoremstyle{plain}
\newtheorem{thm}{Theorem}
\newtheorem{theorem}[thm]{Theorem}
\newtheorem{lemma}[thm]{Lemma}
\newtheorem{corollary}[thm]{Corollary}
\newtheorem{proposition}[thm]{Proposition}
\newtheorem{addendum}[thm]{Addendum}
\newtheorem{variant}[thm]{Variant}
\newtheorem*{thm0}{Theorem}
\newtheorem*{thm35}{Theorem 3.5}
\theoremstyle{definition}
\newtheorem{construction}[thm]{Construction}
\newtheorem{notations}[thm]{Notations}
\newtheorem{question}[thm]{Question}
\newtheorem{problem}[thm]{Problem}
\newtheorem{remark}[thm]{Remark}
\newtheorem{remarks}[thm]{Remarks}
\newtheorem{definition}[thm]{Definition}
\newtheorem{claim}[thm]{Claim}
\newtheorem{assumption}[thm]{Assumption}
\newtheorem{assumptions}[thm]{Assumptions}
\newtheorem{properties}[thm]{Properties}
\newtheorem{example}[thm]{Example}
\newtheorem{conjecture}[thm]{Conjecture}
\newtheorem{sublemma}[thm]{Sublemma}
\numberwithin{equation}{thm}
\newcommand{\pP}{{\mathfrak p}}
\newcommand{\sA}{{\mathcal A}}
\newcommand{\sB}{{\mathcal B}}
\newcommand{\sC}{{\mathcal C}}
\newcommand{\sD}{{\mathcal D}}
\newcommand{\sE}{{\mathcal E}}
\newcommand{\sF}{{\mathcal F}}
\newcommand{\sG}{{\mathcal G}}
\newcommand{\sH}{{\mathcal H}}
\newcommand{\sI}{{\mathcal I}}
\newcommand{\sJ}{{\mathcal J}}
\newcommand{\sK}{{\mathcal K}}
\newcommand{\sL}{{\mathcal L}}
\newcommand{\sM}{{\mathcal M}}
\newcommand{\sN}{{\mathcal N}}
\newcommand{\sO}{{\mathcal O}}
\newcommand{\sP}{{\mathcal P}}
\newcommand{\sQ}{{\mathcal Q}}
\newcommand{\sR}{{\mathcal R}}
\newcommand{\sS}{{\mathcal S}}
\newcommand{\sT}{{\mathcal T}}
\newcommand{\sU}{{\mathcal U}}
\newcommand{\sV}{{\mathcal V}}
\newcommand{\sW}{{\mathcal W}}
\newcommand{\sX}{{\mathcal X}}
\newcommand{\sY}{{\mathcal Y}}
\newcommand{\sZ}{{\mathcal Z}}
\newcommand{\A}{{\mathbb A}}
\newcommand{\B}{{\mathbb B}}
\newcommand{\C}{{\mathbb C}}
\newcommand{\D}{{\mathbb D}}
\newcommand{\E}{{\mathbb E}}
\newcommand{\F}{{\mathbb F}}
\newcommand{\G}{{\mathbb G}}
\renewcommand{\H}{{\mathbb H}}
\newcommand{\I}{{\mathbb I}}
\newcommand{\J}{{\mathbb J}}
\renewcommand{\L}{{\mathbb L}}
\newcommand{\M}{{\mathbb M}}
\newcommand{\N}{{\mathbb N}}
\renewcommand{\P}{{\mathbb P}}
\newcommand{\Q}{{\mathbb Q}}
\newcommand{\R}{{\mathbb R}}
\newcommand{\SSS}{{\mathbb S}}
\newcommand{\T}{{\mathbb T}}
\newcommand{\U}{{\mathbb U}}
\newcommand{\V}{{\mathbb V}}
\newcommand{\W}{{\mathbb W}}
\newcommand{\g}{{\gamma}}
\newcommand{\bb}{{\beta}}
\newcommand{\as}{{\alpha}}
\newcommand{\id}{{\rm id}}
\newcommand{\rk}{{\rm rank}}
\newcommand{\END}{{\mathbb E}{\rm nd}}
\newcommand{\End}{{\rm End}}
\newcommand{\Hom}{{\rm Hom}}
\newcommand{\Hg}{{\rm Hg}}
\newcommand{\tr}{{\rm tr}}
\newcommand{\Sl}{{\rm Sl}}
\newcommand{\Gl}{{\rm Gl}}
\newcommand{\Cor}{{\rm Cor}}

\newcommand{\SO}{{\rm SO}}
\newcommand{\OO}{{\rm O}}
\newcommand{\SP}{{\rm SP}}
\newcommand{\Sp}{{\rm Sp}}
\newcommand{\UU}{{\rm U}}
\newcommand{\SU}{{\rm SU}}
\newcommand{\SL}{{\rm SL}}
\thanks{This work is supported by the SFB/TR 45 `Periods, Moduli
Spaces and Arithmetic of Algebraic Varieties' of the DFG}
\begin{abstract}
Let $k$ be a perfect field of odd characteristic $p$ and $X_0$ a
smooth connected algebraic variety over $k$ which is assumed to be
$W_2(k)$-liftable. In this short note we associate a flat bundle to
a nilpotent Higgs bundle over $X_0$ of exponent $n\leq p-1$ via the
exponential function. Presumably, the association is equivalent to
the inverse Cartier transform in \cite{OV} for these Higgs bundles.
However this point has not been verified in the note. Instead, we
show the equivalence of the association with that of \cite{SXZ} in
the geometric case. The construction relies on the cocycle property
of the difference of different Frobenius liftings over $W_2(k)$,
which plays the key role in the proof of $E_1$-degeration of the
Hodge to de Rham spectral sequence of $X_0$ in \cite{DI}.
\end{abstract}

\maketitle


A grand theory in complex geometry is the nonabelian Hodge theory
connecting flat bundles with Higgs bundles. It has been achieved after
works of many people, notably Narasimhan and Seshadri, Uhlenbeck and
Yau, Donaldson, Hitchin, Corlette, Simpson (see \cite{NS},
\cite{UY}, \cite{Do}, \cite{Hi}, \cite{Co}, \cite{Si1}). Recently,
Ogus and Vologodsky (see \cite{OV}) has established an analogue of
the theory in positive characteristic. In this short note, we intend
to use some rudiments in differential geometry to exhibit a
construction of flat bundles for an important class of Higgs
bundles, which is believed to be equivalent to the inverse Cartier
transform in loc. cit.. \\
Let $k$ be a perfect field with odd characteristic $p$ and $X_0$ a smooth connected algebraic variety over $k$ of dimension $d$. We assume that there is a $W_2:=W_2(k)$-lifting of $X_0$. Namely, there exists a smooth $W_2$-scheme $X_1$ whose closed fiber is $X_0$. \\
Let $(E,\theta)$ be a nilpotent Higgs bundle over $X_0$ of exponent $n\leq p-1$. Locally over an open affine subset $U\subset X_0$ with local coordinates $\{t_1,\cdots,t_d\}$, the Higgs field $\theta$ under a local basis of sections of $E$ over $U$ is written as $\theta=\sum\theta_idt_i$, where $\theta_i=\partial_{t_i}\lrcorner\theta$ is a matrix of elements in $\sO_{U}$. The integrability of the Higgs field is then equivalent to the commuting relations $[\theta_i,\theta_j]=0,  1\leq i < j\leq d$, and the nilpotent condition put on the Higgs field means that $\prod_{j=1}^{d}\theta_j^{i_j}=0$ once $\sum_{j=1}^d i_j\geq n+1$.\\
Before carrying out our construction, we first recall some basics on connection and curvature that could be found in any standard book in differential geometry (see e.g. \cite{K}). Let $H$ be a vector bundle over $X_0$ together with a ($k$)-connection $\nabla: H\to H\otimes \Omega_{X_0}$. Take a local basis $e_U$ of sections of $H$ over $U$. Then over $U$, $\nabla=d+A_U$, where $A_U$ is a matrix-valued one form given by   $\nabla(e_U)=A_Ue_U$. More precisely, the formula means that for a local section $s=s_Ue_U$ of $H$ over $U$, where $s_U$ is a row vector with elements in $\sO_{U}$, one has $\nabla (s)=ds_{U}e_U+ s_UA_Ue_U$. Let $A'_{U}$ be the connection one form of $\nabla$ under another local basis $e'_U=Me_U$ of $H$ over $U$, where $M$ is an invertible matrix with entries in $\sO_U$. Then one has the transformation formula of connection forms:
$$
A'_U=dMM^{-1}+MA_UM^{-1}
$$
Finally, the curvature $K$ of $\nabla$ is defined to the composite
$$
H\stackrel{\nabla}{\longrightarrow}H\otimes \Omega_{X_0}\stackrel{\nabla^1}{\longrightarrow}H\otimes \Omega^2_{X_0},
$$
where $\nabla^1(h\otimes \omega)=\nabla(h)\wedge \omega+h\otimes d\omega$. It is $\sO_{X_0}$-linear and under the
local basis $e_U$ of $H$ over $U$ expressed by the formula $K_U=dA_U + A_U\wedge A_U$.
We say that the connection is integrable if its curvature is zero. A vector bundle with an integrable connection is called a flat bundle.\\
Now we proceed to the construction of the association. First of all we take an affine covering $\mathcal{U}=\{ U'_{\alpha}\}$ of $X_1$. Note that over each $U'_{\alpha}$ we can take a lifting $F_{\alpha}: U'_{\alpha}\to U'_{\alpha}$ of the absolute Frobenius $F_0: U_{\alpha}\to U_{\alpha}$ determined by the power $p$ map, where $U_{\alpha}$ is the closed fiber of $U'_{\alpha}$. The composite of $\sO_{U'_{\alpha}}$-morphisms
$$
F_{\as}^*\Omega_{U'_{\as}} \stackrel{dF_{\alpha}}{\longrightarrow}
p\Omega_{U'_{\as}}\stackrel{\frac{1}{[p]}}{\cong}
\Omega_{U_{\alpha}} $$ descends clearly to an
$\sO_{U_{\alpha}}$-morphism $ \frac{dF_{\alpha}}{[p]}:
F_0^{*}\Omega_{U_{\alpha}}\to \Omega_{U_{\alpha}}$. We consider then
the vector bundle $H_{\alpha}:=F_0^*E|_{U_{\alpha}}$ over
$U_{\alpha}$, where $E|_{U_{\alpha}}$ denotes for the restriction of
$E$ over $U_{\alpha}$. Choose a local basis $e_{\alpha}$ of
$E|_{U_{\alpha}}$ and define a connection on $H_{\alpha}$ by the
formula $\nabla_{\as}=d+\frac{dF_{\alpha}}{[p]}(F_0^*\theta_{\as})$,
where $\theta_{\alpha}$ is the Higgs field over $U_{\alpha}$ under
the local basis $e_{\as}$ of $E|_{U_{\alpha}}$.


\begin{proposition}
The connection $\nabla_{\as}$ on $H_{\alpha}$ is well defined and integrable.
\end{proposition}

\begin{proof}
First we verify the well-definedness. Take another local basis $e'_{\alpha}=Me_{\alpha}$ of $E|_{U_{\alpha}}$ and put $\theta'_{\alpha}$ to be the Higgs field under this basis. According to the transformation formula for connection forms, we shall check that
$$
\frac{dF_{\alpha}}{[p]}(F_0^*\theta'_{\as})=(dF_0^{*}M)F_0^{*}M^{-1}+F_0^*M\frac{dF_{\alpha}}{[p]}(F_0^*\theta_{\as})F_0^*M^{-1}.
$$
As $\theta'_{\alpha}=M\theta_{\alpha}M^{-1}$ and $dF_0^*M=0$, the
above equality holds. Next we verify the integrability of
$\nabla_\alpha$. As $\theta$ is integrable, i.e.
$\theta_{\as}\wedge\theta_{\as}=0$, it follows that
$F_0^*(\theta_\alpha)\wedge
F_0^*(\theta_\alpha)=F_0^*(\theta_\alpha\wedge \theta_\alpha)=0$ and
furthermore $$ \frac{dF_{\alpha}}{[p]}(F_0^*\theta_{\as})\wedge
\frac{dF_{\alpha}}{[p]}(F_0^*\theta_{\as})=(\bigwedge^2\frac{dF_{\alpha}}{[p]})(F_0^*\theta_{\as}\wedge
F_0^*\theta_{\as})=0.
$$
It is left to show that
$d(\frac{dF_{\alpha}}{[p]}(F_0^*\theta_{\as}))=0$. This is done by a
local computation: by definition, for $\omega\in
\Omega_{U_{\alpha}}$, $\frac{dF_{\alpha}}{[p]}(F_0^*\omega)=
\frac{1}{[p]}(dF_{\alpha} (F_{\as}^*\omega'))$, where $\omega'\in
\Omega_{U'_{\alpha}}$ is any lifting of $\omega$. Then $d\circ
\frac{dF_{\alpha}}{[p]}(F_0^*\omega)=d\circ
\frac{1}{[p]}(dF_{\alpha} (F_{\as}^*\omega'))=\frac{1}{[p]}(d\circ
dF_{\alpha} (F_{\as}^*\omega'))$. We may write $\omega'=\sum_i
f_idg_i$ for $f_i,g_i\in \mathcal{O}_{U'_{\as}}$. Then
$d(dF_{\alpha}(F_{\as}^*\omega'))=\sum_i
d(dF_{\alpha}(F_{\as}^*f_i))\wedge d(dF_{\alpha}(F_{\as}^*g_i))\in
p^2\Omega^2_{U'_{\as}}=0$. Thus
$d(\frac{dF_{\alpha}}{[p]}(F_0^*\omega))=0$. Clearly, it follows
that $d(\frac{dF_{\alpha}}{[p]}(F_0^*\theta_{\as}))=0$. The
proposition is proved.
\end{proof}


Thus we have defined a local flat bundle
$(H_{\alpha},\nabla_{\alpha})$ over each $U_{\alpha}\in \sU$. In
order to glue them into a global one, we have to provide a set of
invertible matrices $\{G_{\alpha\beta}\}$ with $G_{\alpha\beta}$
defined over $U_{\alpha\beta}:=U_{\alpha}\cap U_{\beta}$ satisfying
\begin{itemize}
    \item [(i)] the bundle gluing condition over $U_{\alpha\beta\gamma}:=U_{\alpha}\cap U_{\beta}\cap U_{\gamma}$
$$
G_{\alpha\beta}\cdot G_{\beta\gamma}=G_{\alpha\gamma},
$$
\item [(ii)] the connection gluing condition over $U_{\as\bb}$
$$ \frac{dF_{\alpha}}{[p]}(F_0^*\theta_{\as})=dG_{\as \bb}G_{\as\bb}^{-1}  + G_{\as\bb}  \frac{dF_{\beta}}{[p]}(F_0^*\theta_{\bb})G_{\as\bb}^{-1}.$$
\end{itemize}
First a lemma:
\begin{lemma}\label{lemma} There are homomorphisms $h_{\as \bb}:F_0^*\Omega_{U_{\alpha\beta}}\rightarrow \mathcal{O}_{U_{\alpha\beta}}$, satisfying the following two properties:
\begin{itemize}
    \item [(i)] $\frac{dF_{\alpha}}{[p]}-\frac{dF_{\beta}}{[p]}=dh_{\alpha\beta}$,
    \item [(ii)] the cocyle condition over $U_{\alpha\beta\gamma}$: $h_{\as\bb} + h_{\bb\g}=h_{\as\g}$.
\end{itemize}
\end{lemma}
 \begin{proof}
Consider the $W_2$-morphism $G_\alpha: Z'_{\alpha}\to U'_{\alpha\beta}:=U'_{\alpha}\cap U'_{\beta}$ sitting in the following Cartesian diagram:
\[
\begin{CD}
 Z'_{\alpha}@> G_{\as}>>  U'_{\as\bb} \\
 @Vj'_{\as}VV       @VVi'_{\alpha}V\\
 U'_{\as} @>F_{\alpha}>> U'_{\as}
 \end{CD}
\]
where $i'_{\as}$ is the natural inclusion. By reduction modulo $p$, we obtain the following Cartesian square
\[
\begin{CD}
  Z_{\as}@>G_0 >>U_{\as\bb}     \\
 @Vj_{\as}VV       @VVi_{\alpha}V\\
 U_{\as} @>F_0>> U_{\as}
 \end{CD}
\]
Thus we see that $Z_{\as}$ is $U_{\as\bb}$ and $G_{\alpha}:
Z'_{\alpha}\to U'_{\alpha}$ is a lifting of the absolute Frobenius
$F_0$ over $U_{\alpha\beta}$. Similarly for $(U'_{\bb},F_{\beta})$,
we have  $G_{\bb}:Z'_{\bb}\rightarrow U'_{\as\bb} $ which is also a
lifting of $F_0:U_{\as\bb}\rightarrow U_{\as\bb}$. Now we apply
Lemma 5.4 \cite{IL} to the pair $(G_{\as}:Z'_{\as}\rightarrow
U'_{\as\bb}, \ G_{\bb}:Z'_{\bb}\rightarrow  U'_{\as\bb})$ of
Frobenis liftings of the absolute Frobenius $F_0$ on $U_{\as\bb}$,
we get the homomorphisms $h_{\as
\bb}:F_0^*\Omega_{U_{\as\bb}}\rightarrow \mathcal{O}_{U_{\as\bb}}$
such that
$\frac{dF_{\alpha}}{[p]}-\frac{dF_{\beta}}{[p]}=dh_{\alpha\beta}$
and $h_{\as\bb}+h_{\bb\g}=h_{\as\g}$.
\end{proof}
By the lemma, we have locally over $U_{\alpha\beta}$ a matrix of functions $h_{\as\bb}(F_0^*\theta_{\as})$. Notice that, because of the assumption on the exponent of the nilpotent Higgs field $\theta$, the matrix of functions $$g_{\as\bb}:=\exp[h_{\as\bb}(F_0^*\theta_{\as})]$$ is well defined and in fact equal to the finite sum $\sum_{i=0}^{n}\frac{(h_{\as\bb}(F_0^*\theta_{\as}))^i}{i!}$. It is clearly invertible with the inverse $\exp[-h_{\as\bb}(F_0^*\theta_{\as})]$. Furthermore, over $U_{\as\bb}$ we have the transition matrix $\{M_{\alpha\beta}\}$ such that $e_{\alpha}=M_{\alpha\beta}e_{\beta}$. As $E$ is globally defined, one has the cocycle condition for $\{M_{\alpha\beta}\}$:
$$
M_{\as\bb}\cdot M_{\bb\g}=M_{\as\g}.
$$
Now we define our gluing matrices for $\{H_{\alpha}\}_{\alpha\in \sU}$ by $G_{\as\bb}:= g_{\as\bb}\cdot F_0^*M_{\as\bb}$. We have then the following:


\begin{theorem}
The local flat bundles $\{(H_{\as},\nabla_{\as})\}_{\alpha\in\sU}$
glue into a global flat bundle $(H,\nabla)$ via the matrices
$\{G_{\as\bb}\}$.
\end{theorem}

\begin{proof} The proof is divided into two steps.\\
{\itshape Step 1: Bundle gluing.} It is to show the cocycle condition $$G_{\as\bb}\cdot G_{\bb\g}=G_{\as\g}.$$
It is a direct computation. First of all,
\begin{eqnarray*}
  g_{\bb\g}&=& \exp[h_{\bb\g}(F_0^*\theta_{\bb})]\\
&=&\exp[h_{\bb\g}(F_0^*(M_{\as\bb}^{-1}\theta_{\as}M_{\as\bb}))]\\
&=&\exp[F_0^*M_{\as\bb}^{-1} h_{\bb\g}(F_0^*\theta_{\as})F_0^*M_{\as\bb}] \\
  &=&F_0^*M_{\as\bb}^{-1}\exp[h_{\bb\g}(F_0^*\theta_{\as})]F_0^*M_{\as\bb}.
\end{eqnarray*}
So it follows that
\begin{eqnarray*}
 G_{\as\bb}\cdot G_{\bb\g}&=&g_{\as\bb}F_0^*M_{\as\bb}\cdot g_{\bb\g} F_0^*M_{\bb\g}\\
 &=&exp[h_{\as\bb}(F_0^*\theta_{\as})]\exp[h_{\bb\g}(F_0^*\theta_{\as})]F_0^*(M_{\as\bb}\cdot M_{\bb\g})\\
 &=& exp[h_{\as\bb}(F_0^*\theta_{\as})]\exp[h_{\bb\g}(F_0^*\theta_{\as})]F_0^*M_{\alpha\gamma}.
\end{eqnarray*}
It follows from the integrability of Higgs field that the two matrices $h_{\as\bb}(F_0^*\theta_{\as})$ and $h_{\bb\g}(F_0^*\theta_{\as})$ commute with each other. Thus we compute further that
\begin{eqnarray*}
 G_{\as\bb}\cdot G_{\bb\g}&=&\exp[(h_{\as\bb}+h_{\bb\g})(F_0^*\theta_{\as})] F_0^*M_{\as\g} \\
 &=&\exp[h_{\as\g}(F_0^*\theta_{\as})]F_0^*M_{\as\g}\\
 &=& G_{\alpha\g}.
\end{eqnarray*}
The second equality follows from Lemma \ref{lemma} (ii).\\
{\itshape Step 2: Connection gluing.} It is to show that the local
connections  $\{\nabla_{\as}\}$ coincide on the overlaps. Recall
that by definition
$\nabla_{\as}(F_0^*e_{\as})=\frac{dF_{\alpha}}{[p]}(F_0^*\theta_{\as})
F_0^*e_{\as}$. Over $U_{\alpha\beta}$, it holds that
\[
  \begin{array}{lll}
\nabla_{\bb}(F_0^*e_{\as})&=& \nabla_{\bb}(G_{\as\bb}F_0^*e_{\bb}) \\
 &=& dG_{\as \bb}F_0^*e_{\bb} +G_{\as\bb} \nabla_{\bb}(F_0^*e_{\bb})\\
&=& [dG_{\as \bb}G_{\as\bb}^{-1}  + G_{\as\bb}
\frac{dF_{\beta}}{[p]}(F_0^*\theta_{\bb})G_{\as\bb}^{-1}](F_0^*e_{\as}).
 \end{array}
\]
As $d(F_0^*M_{\as\bb})=0$ and $\theta_{\bb}=M_{\as\bb}^{-1} \theta_{\as} M_{\as\bb}$, we have further
\[
  \begin{array}{lll}
\nabla_{\bb}(F_0^*e_{\as})&=& [dg_{\as\bb} g_{\as\bb}^{-1}+g_{\as\bb} \frac{dF_{\beta}}{[p]}(F_0^*\theta_{\as})g_{\as\bb}^{-1}]F_0^*e_{\as} \\
 &=&[dh_{\as\bb}(F_0^*\theta_{\as})+\frac{dF_{\beta}}{[p]}(F_0^*\theta_{\as})] F_0^*e_{\as}\\
&=&\frac{dF_{\alpha}}{[p]}(F_0^*\theta_{\as})F_0^*e_{\as}\\
&=&\nabla_{\as}(F_0^*e_{\as}).
 \end{array}
\]
Here the second equality follows from the fact that $g_{\as\bb}$
commutes with $\frac{dF_{\beta}}{[p]}(F_0^*\theta_{\as})$ which is
again because of the integrability of the Higgs field, and the third
equality uses Lemma \ref{lemma} (i). 
\end{proof}

Presumably, the association is equivalent to the inverse Cartier
transform in \cite{OV}. This has not been verified at the moment.
However, we shall show that it is equivalent to that of \cite{SXZ}
in the geometric case. For sake of convenience, we recall briefly
the setup loc. cit. Let $(H,\nabla,Fil,\Phi)$ be an object in the
category $\mathcal{MF}^{\nabla}_{[0,n]}(X)$ with $n\leq p-2$, where
$X$ is a smooth scheme over $W_{n+1}$. Let
$(E,\theta)=Gr_{Fil}(H,\nabla)$ be the associated Higgs bundle and
$(H,\nabla)_0$ (resp. $(E,\theta)_0$) the flat bundle (resp. the
Higgs bundle) in characteristic $p$. In loc. cit., a flat subbundle
$(H_{(G,\theta)},\nabla)$ of $(H,\nabla)_0$ is associated to any
Higgs subbundle $(G,\theta)$ of $(E,\theta)_0$. To distinguish the
notations,  we denote by $(H_{exp},\nabla_{exp})$ the flat bundle by
applying the previous construction to $(G,\theta)$. We claim the
following
\begin{proposition}
Notation as above. Then there is a natural isomorphism of flat
bundles $(H_{exp},\nabla_{exp}) \cong (H_{(G,\theta)},\nabla)$
induced by the relative Frobenius.
\end{proposition}
\begin{proof}
The isomorphism is obtained by gluing a set of local isomorphisms $\{\tilde{\Phi}_{\tilde F_{\as}}\}$ which is constructed as follows:  recall that loc. cit. for any $W_{n+1}$-lifting $(\tilde U_{\alpha},\tilde F_{\alpha})$ of the pair $(U'_{\alpha},F_{\alpha})$, there is a map $\tilde{\Phi}_{\tilde F_{\as}}: F_0^*(E_0|_{U_{\as}})\stackrel{\cong}{\longrightarrow}H_0|_{U_{\as}} $, induced by relative Frobenius $\Phi$. By Proposition 2.2 loc. cit. this isomorphism does not depend on the choice of liftings $\tilde F_{\alpha}$.  Since $H_{(G,\theta)}|_{U_{\as}}$ is defined to be $\tilde{\Phi}_{\tilde F_{\as}}(F_0^*G|_{U_0})$ by loc. cit., $\tilde{\Phi}_{\tilde F_{\as}}$ induces an isomorphism from $H_{\as}$ to $H_{(G,\theta)}|_{U_{\as}}$, which by abuse of notation is denoted by $\tilde \Phi_{F_{\alpha}}$.\\
{\itshape Step 1: Bundle isomorphism.}
By choosing a local basis $e$ of sections of $G$ over $U_{\as\bb}$, we are going to show that over $ U_{\as\bb}$,
$$\tilde{\Phi}_{F_{\as}}(F_0^*e)=g_{\as\bb}\tilde{\Phi}_{F_{\bb}}(F_0^*e).$$
For simplicity of notation, we denote simply by $\theta$ the Higgs field under the basis $e$. For a multi-index $\underline{j}=(j_1,\cdots,j_d)$, we put  $$\theta_{\partial}^{\underline{j}}=(\partial_{t_1}\lrcorner \theta)^{j_1}\cdots (\partial_{t_d}\lrcorner \theta)^{j_d},\ z_l=\left(\frac{ F_{\as}-F_{\bb}}{[p]} \right)(F_0^*t_l), \ z^{\underline{j}}=\prod_{l=1}^d z_l^{j_l}.$$
(Notice that $z_l$ and $z^{\underline{j}}$ differ from the notation loc. cit..) According to Lemma 2.4 loc. cit., we have by abuse of notation
$$
\tilde{\Phi}_{F_{\as}}(F_0^*e )=\left(1+\sum_{|\underline{j}|=1}^{n}F_0^*(\theta_{\partial}^{\underline{j}})\cdot\frac{z^{\underline{j}}}{\underline{j}\,!}\right)\cdot \tilde{\Phi}_{F_{\bb}}(F_0^*e).
$$
So it suffices to show $g_{\as\bb}=1+\sum_{|\underline{j}|=1}^{n}F_0^*(\theta_{\partial}^{\underline{j}})\cdot\frac{z^{\underline{j}}}{ \underline{j}\,!}$.
As $$
h_{\as\bb}(F_0^*\theta)=\sum_{l=1}^d F_0^*(\partial_{t_l}\lrcorner \theta) h_{\as\bb}(F_0^*dt_l)$$
and $h_{\as\bb}(F_0^*dt_l)=\left(\frac{ F_{\as}-F_{\bb}}{[p]} \right)(F_0^*t_l)=z_l$, it follows that
$$
\frac{(h_{\as\bb}(F_0^*\theta ))^i}{i!}=\frac{(\sum_{l=1}^d F_0^*(\partial_{t_l}\lrcorner \theta)z_l)^i}{i!}=\sum_{|\underline{j}|=i}\frac{F_0^*(\theta_{\partial}^{\underline{j}})z^{\underline{j}} }{\underline{j}\, !}.
$$
Recall that $g_{\as\bb}=\sum_{i=0}^{n}\frac{(h_{\as\bb}(F_0^*\theta ))^i}{i!}$, it follows then $g_{\as\bb}=1+\sum_{|\underline{j}|=1}^{n}F_0^*(\theta_{\partial}^{\underline{j}})\cdot\frac{z^{\underline{j}}}{ \underline{j}\,!}$ as wanted.

{\itshape Step 2: Connection isomorphism.}
We need to show that under the above isomorphism, the connection $\nabla_{exp}$ on $H_{exp}$ is equal to the connection $\nabla $ on $H_{(G,\theta)}$. Put $e_\alpha$ to be a local basis of $G$ over $U_{\alpha}$. By Lemma 2.5 and the proof of Proposition 2.6 loc. cit., we have
$$
\nabla[\tilde{\Phi}_{F_{\alpha}}(F_0^*e_\as)]=\left[\sum_{l=1}^d
F_0^*\theta_l
\frac{dF_{\alpha}}{[p]}(F_0^*dt_l)\right]\tilde{\Phi}_{F_\as}(F_0^*e_\as)
$$
As $\nabla_{exp}(F_0^*e_\as)=\left[\sum_{l=1}^d
F_0^*\theta_l\cdot\frac{dF_{\alpha}}{[p]}(F_0^*dt_l)\right]F_0^*e_\as$,
it follows that
$$\tilde{\Phi}_{F_\as}(\nabla_{exp}(F_0^*e_\as))=\left[\sum_{l=1}^d F_0^* \theta_l\frac{dF_{\alpha}}{[p]}(F_0^*dt_l)\right]\tilde{\Phi}_{F_\as}(F_0^*e_\as) =\nabla[\tilde{\Phi}_{F_\as}(F_0^*e_\as)].$$
\end{proof}

{\bf Acknowledgement:} We would like to thank Arthur Ogus heartily for his valuable comments on this note, particularly pointing us an error in the first version of the note.


\begin{thebibliography}{5}
\bibitem{BO}
P. Berthelot, A. Ogus, Notes on crystalline cohomology,
Princeton University Press, Princeton 1978.

\bibitem{Co}
K. Corlette, Flat $G$-bundles with canonical metrics, J. Diff.
Geom., 28 (1988), 361-382.

\bibitem{DI}
P. Deligne, L. Illusie, Rel\`{e}vements modulo $p^2$ et decomposition du complexe de de Rham, Invent. Math. 89 (1987), 247-270.

\bibitem{Do}
S. K. Donaldson, Anti self dual Yang-Mills connections over complex
algebraic surfaces and stable vector bundles, Proc. London Math.
Soc. (3), 50 (1985), 1-26. Infinite determinants, stable bundles,
and curvature, Duke Math. J., 54 (1987), 231-247. Twisted harmonic
maps and self-duality equations, Proc. London Math. Soc., 55 (1987),
127-131.

\bibitem{EV}
H. Esnault, E. Viehweg, Lectures on vanishing theorems. DMV Seminar, 20. Birkh\"{a}user Verlag, Basel, 1992.

\bibitem{NS}
M. S. Narasimhan, C. S. Seshadri, Stable and unitary vector bundles on a compact Riemann surface, Ann. of Math., 82, 1965, 540-567.

\bibitem{Hi}
N. J. Hitchin, The self-duality equations on a Riemann surface,
Proc. London Math. Soc., 55 (1987), 59-126.

\bibitem{IL}
L. Illusie, Frobenius et d\'{e}g\'{e}n\'{e}rescence de Hodge, Introduction \`{a} la th\'{e}orie de Hodge, 113-168, Panor. Synth\`{e}ses, 3, Soc. Math. France, Paris, 1996.

\bibitem{K} S. Kobayashi, Differential geometry of complex vector bundles, Publications of the Mathematical Society of Japan, 15, Kan\^{o} Memorial Lectures 5, 1987.

\bibitem{SXZ}M. Sheng, H. Xin, K. Zuo,
A note on the characteristic $p$ nonabelian Hodge theory in the geometric case, arXiv:1202.3942v2, 2012.

\bibitem{Si1}
C. Simpson, Higgs bundles and local systems, Publ. Math. Inst.
Hautes \'{e}tud. Sci. 75 (1992), 5-95.

\bibitem{OV}
A. Ogus, V. Vologodsky, Nonabelian Hodge theory in characteristic $p$,
Publ. Math. Inst. Hautes \'{e}tudes Sci. 106 (2007), 1-138.


\bibitem{UY}
K. K. Uhlenbeck, S.-T. Yau, On the existence of Hermitian-Yang-Mills
connections in stable vector bundles, Comm. Pure Appl. Math. 39
(1986), 257-293.
\end{thebibliography}
\end{document}